\documentclass[12pt]{amsart}
\usepackage[utf8]{inputenc}
\usepackage{amsfonts}
\usepackage{amssymb}
\usepackage{amsmath}
\usepackage[english]{babel}
\usepackage{amsthm}
\theoremstyle{plain}
\newtheorem{theorem}{Theorem}[section]
\newtheorem{corollary}[theorem]{Corollary}
\newtheorem{lemma}[theorem]{Lemma}
\newtheorem{proposition}[theorem]{Proposition}
\theoremstyle{definition}
\newtheorem{definition}[theorem]{Definition}
\newtheorem{remark}[theorem]{Remark}
\usepackage{geometry}
\numberwithin{equation}{section} 
\usepackage{tikz}
\usepackage{caption}
\usepackage{graphicx}

\usepackage[colorlinks, citecolor=blue,pagebackref,hypertexnames=false]{hyperref}

\newcounter{comcount}

\begin{document}
\title{ END POINT ESTIMATES FOR RIESZ TRANSFORM and Hardy-Hilbert type inequalities} 
\date{}
\author{ Dangyang He}  
\maketitle

\begin{abstract}
     We consider a class of non-doubling manifolds $\mathcal{M}$ defined by taking connected sum of finite Riemannian manifolds with dimension N which has the form $\mathbb{R}^{n_i}\times \mathcal{M}_i$ and the Euclidean dimension $n_i$ are not necessarily all the same. In \cite{HaSi}, Hassell and Sikora proved  that  the  Riesz transform on $\mathcal{M}$  is weak type $(1,1)$,  bounded on $L^{p}(\mathcal{M})$ for all $1<p<n^*$ where $n^*  = \min_k n_k$ and is unbounded for $p \ge n^*$.
     In this note we show that the Riesz transform is bounded from Lorentz space $L^{n^* ,1}(\mathcal{M})$ to $L^{n^*,1}(\mathcal{M})$. This complete the picture by obtaining the end point results for $p=n^*$.

     Our approach is based on parametrix construction described in \cite{HaSi} and a generalisation of Hardy-Hilbert type inequalities first studied by Hardy, Littlewood and P\'olya.
\end{abstract}

\tableofcontents
\section{Introduction}

Riesz transform is an important area of Harmonic Analysis. It provides answer to a very natural and significant questions. Arguably it is also an origin of the whole modern development of Harmonic Analysis especially the theory of singular integrals, which Riesz transform is par excellence example. To explain the significance of Riesz transform question we start with some simple comparison of the two-dimensional wave and Laplace equations. For any two functions $f,g$ which map $\mathbb{R}$
to itself the corresponding function $F(x,t)=f(x-t)+g(x+t)$
satisfy the wave equation $\partial_t^2F-\partial_x^2F=0$ as it is present in D'Alembert's Formula. As a consequence, potential solutions do not have any regularity and the components 
$\partial_t^2F$ and $\partial_x^2F$ are not necessarily meaningful. 
On the other hand, for Laplace equation if $\partial_t^2F+\partial_x^2F=u$ for some $u\in L^p(\mathbb{R})$
then the boundedness  of the second-order Riesz transforms 
$\partial_t^2/(\partial_t^2+\partial_x^2)$ and 
$\partial_x^2/(\partial_t^2+\partial_x^2)$ implies that both derivatives $\partial_t^2F$ and $\partial_x^2F$
are meaningful and in fact they belong to the same $L^p$
space for all $1<p<\infty$. It is one of the most remarkable and significant observations in the theory of Partial differential equations, especially surprising in the above comparison with the wave equation.  
For $p=2$ the boundedness of the Riesz transform of any order is a direct consequence of the Plancherel's Theorem and theory of Lebesgue integration. Whereas the results for $p\neq 2$ opened a new chapter in modern Harmonic analysis called singular integrals. 

The first order Riesz transform is defined as 
\begin{equation*}
    R_i = \frac{\partial_i}{\sqrt{\Delta}} \quad 1\le i\le n
\end{equation*}
where $\Delta = -\sum_{i=1}^n \partial_i^2$ is the standard Laplacian on $\mathbf{R}^n$. One of applications of the study of Riesz transform is the definition of Sobolev spaces. In a simple example, the continuity of $R_i$ on $L^p$ spaces shows that the following two natural definitions of homogeneous Sobolev spaces are equivalent:
\begin{equation*}
    \|f\|_{\Dot{W}^{1,p}} = \|\Delta^{1/2}f\|_p \sim \sum_{i=1}^n \|\partial_i f\|_p \quad 1<p<\infty
\end{equation*}

The first momentum of the theory of the Riesz transform 
comes from its one-dimension version, Hilbert transform, which appears in the way:
\begin{equation*}
    H(f) = \lim_{y\to 0}(f*Q_y)
\end{equation*}
where $Q_y$ is the conjugate Poisson kernel and the limit is taking in $L^2$ \cite[Proposition 3.1]{SS}, and the proof of its boundedness was described by Riesz \cite{15} based on complex analysis. In higher dimensional setting, the extension of the $L^p$ boundedness result of Riesz transform was described in the way of Calderón–Zygmund decomposition \cite{S}.

In \cite{16}, Strichartz initiated the project to extend the classical results known for the standard Laplace operator to the setting of Laplace-Beltrami operator on the complete Riemannian manifolds. Such generalizations hold automatically for some type of results but in the case of Riesz transform, the question is still not completely solved. Consider the Riesz transform on a complete Riemannian manifold equipped with some measure satisfying doubling condition. That is
\begin{equation}
    \mu(B(x,2r))\le C \mu(B(x,r)) \quad \forall x,r
\end{equation}
where $B(x,r)$ is the geodesic ball centered at $x$ with radius $r$. Coulhon and Duong proved that the Riesz transform on this kind of manifold is bounded in $L^p$ for $1<p\le 2$ and of weak type $(1,1)$, the assumptions and details can be found in \cite{6}.

We also consider the case where the doubling condition fails. In \cite{Ca, CaCo, GrSC, HaSi} authors focus on a special class of manifolds which is called the connected sums of products of Euclidean spaces and compact manifolds. The main settings can be expressed as follows. Let $\mathcal{M}$ be a manifold of connected sum in the form
\begin{equation*}
    \mathcal{M} = (\mathbb{R}^{n_1} \times \mathcal{M}_1)\# \dots \#(\mathbb{R}^{n_l} \times \mathcal{M}_l)
\end{equation*}
with $l\ge 2$, $n_i\ge 3$ and each $\mathcal{M}_i$ is a compact manifold. Also, let $N = n_i + \mathop{dim}\mathcal{M}_i$ to be the topological dimension of $\mathcal{M}$. Denote by $\Delta$ the Laplace-Beltrami operator, $\nabla$ the gradient on $\mathcal{M}$.  The Riesz transform on $\mathcal{M}$ is then defined as
\begin{equation*}
    R = \nabla \Delta^{-\frac{1}{2}}
\end{equation*}
then the main question is to determine for which exponent $p$, the Riesz transform $R$ is bounded from $L^{p}(\mathcal{M})$ to $L^{p}(\mathcal{M})$. In \cite{HaSi}, Hassell and Sikora proved the following theorem.

\begin{theorem}\label{thmHA}
Suppose that $\mathcal{M} = (\mathbb{R}^{n_1} \times \mathcal{M}_1)\# \dots \#(\mathbb{R}^{n_l} \times \mathcal{M}_l)$ is a manifold with $l \ge  2$ ends, and $n_i\ge 3$ for each $i$. Then the Riesz transform $\nabla \Delta^{-\frac{1}{2}}$ defined on $\mathcal{M}$ is bounded on $L^p(\mathcal{M})$ if and only if $1 < p < \min \{n_1, \dots, n_l\}$. That is,
there exists $C>0$ such that
\begin{equation*}
    \|\nabla \Delta^{-\frac{1}{2}}(f)\|_p \le C\|f\|_p
\end{equation*}
if and only if $1 < p < n^*$ where $n^* = \min \{n_1, \dots, n_l\}$. In addition the Riesz transform $\nabla \Delta^{-\frac{1}{2}}$ is of
weak type (1,1).
\end{theorem}

This result improves the work of Carron \cite{Ca} and gives an example of a setting in non-doubling manifold when $n_i$ are not all equal. Together with weak type (1,1) estimates, the $L^p$ boundedness of the Riesz transform has been described except the estimate for the end point $n^* = \min_k n_k$. 

To complete the picture, we wish to give a suitable endpoint result. In \cite[Section 5]{6}, Coulhon and Duong give a counterexample to show that $R$ is not bounded in $L^p$ for $p>n^*$ in the case where $n_i=n_j$. From \cite[Proposition 6.1]{HaSi}, we know that in a more general setting where $n_i$ does not need to equal to $n_j$, the strong type estimate cannot hold for $p\ge n^*$. In \cite[Section 3.7]{Li}, in a different setting, Li gives a counterexample to illustrate that the Riesz transform $R$ is not weak type $(p_0,p_0)$ for some critical value $p_0$ (In our case $p_0=n^*$). Therefore, it is natural for us to consider whether the restricted weak type estimates can be established.
 
\bigskip

To describe the endpoint extension of Theorem \ref{thmHA}, we have to recall notions in Lorentz space.
By $f^*$  we denote the decreasing rearrangement function of $f$, defined by the formula 
\begin{equation*}
    f^*(t) = \inf\{s>0; d_f(s)\le t\}
\end{equation*}
and $d_f$ is the usual distribution function of $f$ i.e.
\begin{equation*}
    d_f(s) = \mu(\{x\in X\colon |f(x)| >s  \}).
\end{equation*}

\begin{definition}
Let $f$ be a measurable function defined on a measure space $(X,\mu)$. For $0<p,q\le \infty$, define
\begin{equation*} 
     \|f\|_{(p,q)}=  
    \begin{cases}
    \left( \int_0^{\infty} \left(t^{1/p} f^*(t)\right)^q \frac{dt}{t}   \right)^{\frac{1}{q}}, & q<\infty\\
     \sup_{t>0} t^{1/p}f^*(t), & q=\infty
    \end{cases}
\end{equation*}
then we say $f$ is in the Lorentz space $L^{p,q}(X,\mu)$ if $\|f\|_{(p,q)}<\infty$.
\end{definition}

We mention that when $q=\infty$, the Lorentz space $L^{p,\infty}$ coincides with the weak $L^p$ space and
\begin{equation*}
    \sup_{t>0} t^{1/p}f^*(t) =  \sup_{\alpha>0} \alpha d_f(\alpha)^{1/p}
\end{equation*}
For a detailed discussion of the notion of Lorentz spaces, see \cite[section 1.4, Page 48]{8}.

Now we are in a position to state our main result, which is described in the following theorem. Our approach is based on the setting and techniques in \cite{9, GrSC, HaSi}.

\begin{theorem}\label{thmDYH}
Let $\mathcal{M} = (\mathbb{R}^{n_1} \times \mathcal{M}_1)\# \dots \#(\mathbb{R}^{n_l} \times \mathcal{M}_l)$ be a manifold with $l \ge  2$ ends, and $n_i\ge 3$ for each $i$. The Riesz transform, $R=\nabla \Delta^{-\frac{1}{2}}$, is bounded in $L^{n^* ,1}(\mathcal{M})$. That is, there exists some $C>0$ such that
\begin{equation}\label{1.2e}
    \|\nabla \Delta^{-\frac{1}{2}}(f)\|_{(n^*,1) }\le C\|f\|_{(n^* ,1)}
\end{equation}
where $n^*  = \min_{1\le k\le l}n_k$.

Moreover, $R$ is not bounded from $L^{n^*,p}\to L^{n^*,p}$ for any $p>1$.

\end{theorem}

We mention that this result is slightly stronger than the restricted weak type estimate i.e.
\begin{equation*}
    \|\nabla \Delta^{-\frac{1}{2}}(f)\|_{(n^*,\infty)}\le C\|f\|_{(n^* ,1)}
\end{equation*}
And the unboundedness in $L^{n^*,p}$ indicates it is not weak type $(n^*,n^*)$. Compare with \cite{Li}.

Some parts of the proof of boundedness of Riesz transform from \cite{HaSi} for $1<p<\infty$ can be partially simplified if one has only proved weak type $(1,1)$ estimate. Since we can use interpolation between Theorem \ref{thmDYH} and weak $(1,1)$ the get the boundedness for $p$ between. See Remark \ref{reinterpolation} below.

The subject of the Riesz transform is so broad that it is
impossible to provide a comprehensive bibliography of it here.
We refer readers to \cite{acdh, bak2, Ca, CaCo, 6, CJKS, G1, hl2, Si, shz05} and references within for the discussion of some other aspects of the Riesz transform theory.


\section{Manifolds with ends}
We start our discussion by introducing the notion of connected sum and manifolds with ends. We refer readers to \cite{GIS, 9, GrSC} for more detailed discussion of the notion of manifolds with ends.

\begin{definition}\label{2.1d}
Let $M_i$ where $i = 1,2,\dots,l$ be a family of complete connected non-compact Riemannian manifolds with the same dimension. Then we call the Riemannian manifold $\mathcal{M}$ is the connected sum of $M_1, . . . , M_l$ and write it as
\begin{equation*}
    \mathcal{M} = M_1\#  M_2\# \dots \# M_l
\end{equation*}
if for some compact subset $K \subset \mathcal{M}$, its compliment $\mathcal{M} \setminus K$ is a disjoint union of connected open sets $E_i$ where $i = 1,2,\dots,l$ such that each $E_i$ is isometric to $M_i \setminus K_i$ for some compact set $K_i \subset  M_i$. We call the subsets $E_i$ the $i$th end of $\mathcal{M}$. In other words, we cut holes, $K_i$, on each of the $M_i$ and then glue $M_i \setminus K_i$ together with the compact set $K$.
\end{definition}
In the following, we use $\mathcal{M}$ to denote manifold:
\begin{equation*}
    \mathcal{M} = (\mathbb{R}^{n_1} \times \mathcal{M}_1)\# \dots \#(\mathbb{R}^{n_l} \times \mathcal{M}_l)
\end{equation*}
which consists of the compact connection $K$ and $l$ ends $E_i=(\mathbb{R}^{n_i}\times \mathcal{M}_i) \setminus K_i$. For simplicity, we may assume that each $K_i$ is the product of $\mathcal{M}_i$ and a closed ball $\overline{B(x_i^0 , \epsilon_i)}\subset \mathbb{R}^{n_i}$ for some $x_i^0 \in \mathbb{R}^{n_i}$ and $\epsilon_i>0$. Moreover, we define the Riemannian metric on each end to be the product metric. That is $\forall z,z'\in (\mathbb{R}^{n_i}\times \mathcal{M}_i) \setminus K_i$, say $z = (x,y)$ and $z' = (x',y')$ where $x,x'\in \mathbb{R}^{n_i} \setminus \overline{B(x_i^0, \epsilon_i)}$ and $y,y'\in \mathcal{M}_i$ we have
\begin{equation*}
    d_{\mathcal{M}}(z,z')^2 = d_{\mathbb{R}^{n_i}}(x,x')^2 + d_{\mathcal{M}_i}(y,y')^2
\end{equation*}
where $d_{\mathcal{M}}, d_{\mathbb{R}^{n_i}}, d_{\mathcal{M}_i}$ are the corresponding Riemannian distance on $\mathcal{M}, \mathbb{R}^{n_i}, \mathcal{M}_i$. For simplicity, in what follow we will use $d(z,z')$ instead of $d_{\mathcal{M}}(z,z')$. We also introduce $\mu_i$, the induced Riemannian measure on $E_i$ and $\mu$, the measure on $\mathcal{M}$ with identity $\mu |_{E_i} = \mu_i$.

Observe that since $\mathcal{M}_i$ is compact, the Riemannian metric locally behaves like in $\mathbb{R}^{N}$ but behaves like $\mathbb{R}^{n_i}$ globally. That is for any geodesic ball $B(z,r)\subset E_i$ centered at $z$ with radius $r$, we have its volume estimate:
\begin{equation*}
    \mu_i(B(z,r)) \sim 
    \begin{cases}
        r^N, & r\le 1\\
        r^{n_i}, & r\ge 1
    \end{cases}
\end{equation*}
Note that the manifolds satisfies the doubling condition if and only if $n_i=n_j$ for any $1\le i,j\le l$.

\subsection{Resolvent of the Laplacian}

 We follow techniques described  in  \cite{HNS, HaSi, 14} and  briefly summarize the key idea of the approach and emphasize the crucial estimates and lemmas we use to prove Theorem \ref{thmDYH}. The starting point is primarily based on the study of the resolvent of the Laplacian. Since the Laplace operator $\Delta$ is positive and self-adjoint, the spectral theory guarantees that
\begin{equation} \label{2.1e}
    \nabla \Delta^{-\frac{1}{2}} = \frac{2}{\pi}\nabla \int_0^{\infty}(\Delta + k^2)^{-1}dk 
\end{equation}
Then the idea is to split it into two parts
\begin{equation} \label{2.2e}
    \nabla \Delta^{-\frac{1}{2}} = \frac{2}{\pi}\nabla \int_0^{1}(\Delta + k^2)^{-1}dk + \frac{2}{\pi}\nabla \int_{1}^{\infty}(\Delta + k^2)^{-1}dk 
\end{equation}
and we define
\begin{gather*}
    R_L = \frac{2}{\pi}\nabla \int_0^{1}(\Delta + k^2)^{-1}dk\\
    R_H = \frac{2}{\pi}\nabla \int_{1}^{\infty}(\Delta + k^2)^{-1}dk
\end{gather*}

We call $R_L$ and $R_H$ the low and high energy part of $R$. In the view of \cite[Proposition 6.1]{HaSi}, we know $R_H$ is bounded in $L^p$ for all $1<p<\infty$. Then by using Theorem \ref{thmLSI} below with $r=1$, we get $R_H$ is bounded in $L^{n^*,1}$. Therefore, instead of proving Theorem \ref{thmDYH}, it suffices to show the end point result for low energy part, $R_L$. The result of \cite[Proposition 6.1]{HaSi} can be stated in the following.
\begin{proposition}\label{propHA}
The Riesz transform localized to high energies, $R_H$, is bounded on $L^p({\mathcal{M}})$ for $1<p<\infty$. In addition, it is of weak type (1,1), that is, it is bounded from $L^1({\mathcal{M}})$ to $L^{1,\infty}({\mathcal{M}})$.
\end{proposition}

\subsection{The estimates of resolvent on the product space $\mathbb{R}^{n_i}\times \mathcal{M}_i$}
According to \eqref{2.1e} and \eqref{2.2e}, we can see that the resolvent plays an important role in constructing our Riesz transform operator. Therefore, in order to give an appropriate estimate for the kernel of $R_L$, we first give some straight estimates for the kernel of resolvent, also see \cite{HNS, HaSi, 14}. Denote by $\Delta_{\mathbb{R}^{n_i}\times \mathcal{M}_i}$ and $e^{-t \Delta_{\mathbb{R}^{n_i}\times \mathcal{M}_i}}(z,z')$ the Laplacian and heat kernel on $\mathbb{R}^{n_i}\times \mathcal{M}_i$ where we use the coordinates $z = (x,y)$, $x\in \mathbb{R}^{n_i}$ and $y\in \mathcal{M}_i$. 

By using the identity in heat semi-group 
\begin{equation} \label{2.3e}
    (\Delta_{\mathbb{R}^{n_i}\times \mathcal{M}_i}+k^2)^{-1} = \int_0^{\infty}e^{-t\Delta_{\mathbb{R}^{n_i}\times \mathcal{M}_i}}e^{-k^2t}dt
\end{equation}

We divert our focus on the estimates for the heat kernel. It is well-known that the heat kernel on the normal Euclidean space has expression
\begin{equation*}
    e^{-t \Delta_{\mathbb{R}^{n_i}}}(x,x') = \frac{1}{(4\pi t)^{n_i/2}} \textrm{exp}\left(-\frac{|x-x'|^2}{4t}\right)
\end{equation*}
where $x,x'\in \mathbb{R}^{n_i}$ and $t>0$.

While for the heat kernel on the compact manifold, $\mathcal{M}_i$, we have Gaussian estimates for $0<t\le 1$
\begin{equation*}
    C t^{-\frac{N-n_i}{2}} \textrm{exp}\left(\frac{d(y,y')^2}{C t}\right)\le e^{-t \Delta_{\mathcal{M}_i}}(y,y')\le C' t^{-\frac{N-n_i}{2}} \textrm{exp}\left(\frac{d(y,y')^2}{C' t}\right)
\end{equation*}
and for $t\ge 1$
\begin{equation*}
    C \le e^{-t \Delta_{\mathcal{M}_i}}(y,y')\le C'
\end{equation*}
where $y,y'\in \mathcal{M}_i$ and $0<C<C'<\infty$. See \cite{ChengLi, ChYa}. 

Since the heat kernel on the product space is just the product of heat kernel on each of the factors. Hence, we have
\begin{equation*}
    e^{-t \Delta_{\mathbb{R}^{n_i}\times \mathcal{M}_i}}(z,z') = e^{-t \Delta_{\mathbb{R}^{n_i}}}(x,x')e^{-t \Delta_{\mathcal{M}_i}}(y,y')
\end{equation*}
thus,
\begin{equation}\label{2.4e}
    C (t^{-\frac{n_i}{2}} + t^{-\frac{N}{2}})\textrm{exp}\left(-\frac{d(z,z')^2}{ct}\right)\le e^{-t \Delta_{\mathbb{R}^{n_i}\times \mathcal{M}_i}}(z,z')\le  C' (t^{-\frac{n_i}{2}} + t^{-\frac{N}{2}})\textrm{exp}\left(-\frac{d(z,z')^2}{c't}\right)
\end{equation}

Now we consider ordinary differential equation
\begin{equation*}
    f''(r) + \frac{a -1}{r}f'(r) =f(r)
\end{equation*}
there exists a positive solution
\begin{equation*}
    L_{a}(r) = r^{1-a/2}K_{|a/2-1|}(r)
\end{equation*}
where $K$ is the modified Bessel function \cite[Page 374]{AbMi}. Moreover, the asymptotics for $L_a$ when $a \ge 3$ are given by
\begin{equation*}
    C r^{2-a} e^{-r} \le L_a(r) \le C'r^{2-a}e^{-cr}, \quad 0<c<1
\end{equation*}
notice that for all $a \ge 1$
\begin{equation*}
    \int_0^{\infty}e^{-tk^2}t^{-\frac{a}{2}}e^{-r^2/(4t)}dt = C_{a}k^{a -2}L_a(kr)
\end{equation*}

By using the asymptotics above, the estimates for the kernel of the resolvent are given by
\begin{gather}\nonumber
    (\Delta_{\mathbb{R}^{n_i}\times \mathcal{M}_i} + k^2)^{-1}(z,z') = \int_0^{\infty}e^{-k^2t}e^{-t\Delta_{\mathbb{R}^{n_i}\times \mathcal{M}_i}}(z,z')dt \\ \nonumber
    \le C \int_0^{\infty}e^{-tk^2}(t^{-\frac{n_i}{2}} + t^{-\frac{N}{2}})\textrm{exp}\left(-\frac{d(z,z')^2}{ct}\right)dt \\= \nonumber
    C (k^{n_i-2}L_{n_i}(ckd(z,z'))+k^{N-2}L_{N}(ckd(z,z'))) \\
    \le C(d(z,z')^{2-N} + d(z,z')^{2-n_i})e^{-ckd(z,z')}
    \label{2.5e}
\end{gather}
and similarly
\begin{equation}\label{2.6e}
    (\Delta_{\mathbb{R}^{n_i}\times \mathcal{M}_i} + k^2)^{-1}(z,z')\ge C'(d(z,z')^{2-N} + d(z,z')^{2-n_i})e^{-ckd(z,z')}
\end{equation}

\subsection{Low energy parametrix}
In this section, we briefly introduce the methods used in \cite{11, HaSi} to construct the parametrix of the kernel of the resolvent $(\Delta +k^2)^{-1}$ on $\mathcal{M}$ in low energy.

At first for each $\mathbb{R}^{n_i} \times \mathcal{M}_i$ $(1\le i\le l)$, we fix a point $z_i^0\in K_i$. Referring to \cite[Lemma 2.7]{HaSi} and \cite[Lemma 3.19]{14}, we mention the following crucial lemma.

\begin{lemma}\label{leKey}
Assume that each $n_i \ge 3$. Let $v \in C_c^{\infty}(\mathcal{M};\mathbb{R})$. Then there is a function $u: \mathcal{M} \times \mathbb{R}^+ \to  \mathbb{R}$ such that $(\Delta + k^2)u = v$ and on the $i$th end we have:
\begin{equation}\label{2.7e}
    |u(z, k)| \le  C d(z_i^0 , z)^{2-n_i} e^{- ck d(z_i^0, z)}\quad  \forall z \in  \mathbb{R}^{n_i}\times \mathcal{M}_i
\end{equation}
\begin{equation}\label{2.8e}
    |\nabla u(z, k)| \le  C d(z_i^0 , z)^{1-n_i} e^{- ck d(z_i^0, z)}\quad  \forall z \in  \mathbb{R}^{n_i}\times \mathcal{M}_i
\end{equation}
\begin{gather*}
    |u(z, k)| \le  C \quad  \forall z \in K\\
    |\nabla u(z, k)| \le  C \quad  \forall z \in K
\end{gather*}
for some $c,C > 0$.
\end{lemma}
Following \cite{HaSi}, we construct $G(k)$, the parametrix of $(\Delta +k^2)^{-1}$ on $\mathcal{M}$ in low energy, as follows. Pick $\phi_i \in C^{\infty}(\mathcal{M})$ such that the support of $\phi_i$ is contained in $E_i$ and it equals 1 outside some compact subset of $\mathbb{R}^{n_i}\times \mathcal{M}_i$. Then we set $v_i = -\Delta\phi_i$ which is clearly supportted on some annuli of $\mathbb{R}^{n_i}\times \mathcal{M}_i$. Moreover we let $G_{int}(k)$ be an interior parametrix of resolvent which is supported in a set close to
 $K\times K$, say $O$. And $G_{int}(k)=(\Delta_{\mathbb{R}^{n_i}\times \mathcal{M}_i}+k^2)^{-1}(z,z')$ on $O\cap \textrm{supp}(\nabla \phi_i(z)\phi_i(z'))$. We mention here $\mathcal{M}^2$ is consisting of pieces: $E_i\times E_j$, $E_i\times K$, $K\times E_j$ and $K\times K$ for $1\le i,j\le l$. Next, we let $u_i(z,k)$ be a function defined on $\mathcal{M}\times \mathbb{R}^+$ such that $(\Delta +k^2)u_i = v_i$. Let $z_i^0$ be some fixed point in $K_i$. Then the parametrix $G(k)$ is defined as 
\begin{equation*}
    G(k) = G_1(k) + G_2(k) + G_3(k) +G_4(k)
\end{equation*}
where
\begin{equation}\label{2.9e}
    G_1(k) = \sum_1^l (\Delta_{\mathbb{R}^{n_i}\times \mathcal{M}_i} +k^2)^{-1}(z,z') \phi_i(z) \phi_i(z')
\end{equation}
\begin{equation}\label{2.10e}
    G_2(k) = G_{int}(k)  \left(1-\sum_1^l \phi_i(z)\phi_i(z')\right)
\end{equation}
\begin{equation}\label{2.11e}
    G_3(k) = \sum_1^l (\Delta_{\mathbb{R}^{n_i}\times \mathcal{M}_i} +k^2)^{-1}(z_i^0,z') u_i(z,k)\phi_i(z')
\end{equation}
Observe that $G_1$ is supported in the "diagonal", $\bigcup_{1\le i\le l} E_i\times E_i$. $G_2$ is compactly supported in $\mathcal{M}^2$ and $G_3$ is defined on whole $\mathcal{M}^2$. Finally, $G_4$ is the error term which will be defined as follows.

First, we define error term $E(k)$ by
\begin{equation*}
    (\Delta + k^2)(G_1(k)+G_2(k)+G_3(k)) = I + E(k)
\end{equation*}
which can be calculated explicitly as
\begin{gather*}
    E(k) = \sum_{i = 1}^l \Big( -2 \nabla \phi_i(z) \phi_i(z') \big(\nabla_z (\Delta_{\mathbb{R}^{n_i} \times \mathcal{M}_i} + k^2)^{-1}(z,z') - \nabla_z G_{int}(z,z')\big)\\
    +\phi_i(z') v_i(z) \big( -(\Delta_{\mathbb{R}^{n_i} \times \mathcal{M}_i} + k^2)^{-1}(z,z') + G_{int}(z,z') + (\Delta_{\mathbb{R}^{n_i} \times \mathcal{M}_i} + k^2)^{-1}(z_i^0,z')\big) \Big)\\
    + \big( (\Delta + k^2)G_{int}(z,z') - \delta_{z'}(z)      \big) \left(1-\sum_1^l \phi_i(z)\phi_i(z')\right)
\end{gather*}
Note that $E(k)(z,z')$ is smooth and compactly supported in the first variable where precisely 
\begin{equation*}
    \textrm{supp}(E(z,z')) \subset  \left(\bigcup_{i=1}^l \textrm{supp}(\nabla \phi_i) \right) \times \mathcal{M} \subset \mathcal{M}^2
\end{equation*}
Moreover if we let $\eta(z)$ be the characteristic function of the set $\bigcup_{i=1}^l \rm{supp}(\nabla \phi_i)$. By \eqref{2.5e}, we have
\begin{equation}\label{2.12e} 
    |E(k)(z,z')| \lesssim  
    \begin{cases}
    \eta(z), & z'\in K\\
     \eta(z) d(z_j^0,z')^{1-n_j}  \textrm{exp} \left(-ck \cdot d(z_j^0,z')\right), & z' \in E_j 
    \end{cases}
\end{equation}
Now for each $z' \in \mathcal{M}$ we define function $v_{z'}(z) = E(k)(z,z')$ which is in $C_c^{\infty}(\mathcal{M})$. By Lemma \ref{leKey}, we can set
\begin{equation*}
    G_4(k)(z,z') = -(\Delta +k^2)^{-1} v_{z'}(z)
\end{equation*}
and thus
\begin{equation*}
    (\Delta+k^2)G(k) = I
\end{equation*}
which implies that the parametrix $G(k)$ is actually the exact resolvent of $\Delta$ in low energy. 

Moreover, Lemma \ref{leKey} also tells us $\forall z\in \mathbb{R}^{n_i}\times \mathcal{M}_i$,
\begin{gather}\label{2.13e}
    |\nabla G_4(k)(z,z')| \lesssim \\
    \begin{cases}
    1, & z\in K, z'\in K\\ \nonumber
    d(z_j^0,z')^{1-n_j} \textrm{exp}(-ck\cdot d(z_j^0,z')), & z\in K, z'\in E_j\\
    d(z_i^0,z)^{1-n_i}\textrm{exp}(-ck\cdot d(z_i^0,z)), & z\in E_i, z'\in K\\
    d(z_j^0,z')^{1-n_j} d(z_i^0,z)^{1-n_i} \textrm{exp}(-ck\cdot d(z_j^0,z')) \textrm{exp}(-ck\cdot d(z_i^0,z)), & z\in E_i, z'\in E_j
    \end{cases}
\end{gather}
where $1\le i,j \le l$.

\section{Hardy-Hilbert type inequalities and Lorentz spaces}
In this section, we prove Hardy-Hilbert type inequalities by using the tools in Lorentz spaces. In section 3.2 and 3.3, we generalize Hardy-Littlewood-P\'olya inequality with settings of homogeneous space and inhomogeneous space separately when operator acting between spaces with different "dimensions". Finally, we use these results to give a strong type $(p,q)$ picture for part of $R_L$ as an application. 

In the case of $\mathbb{R}$, Riesz transform is equivalent with Hilbert transform. In \cite{15}, Riesz proved that Riesz transform on one-dimension is bounded on $L^p$ for all $1<p<\infty$ which motivated some further studies like the boundedness of the Hilbert transform on half-line:
\begin{equation*}
f \mapsto \int_0^\infty \frac{f(y)}{x+y}dy    
\end{equation*}
and the famous Hardy's inequality:
\begin{equation*}
    \|\frac{1}{x} \int_0^x |f(t)|dt\|_{L^p(\mathbb{R}^+,dx/x)} \le \frac{p}{p-1} \|f\|_{L^p(\mathbb{R}^+,dx/x)} \quad 1<p<\infty
\end{equation*}

In \cite[Theorem 319]{HLP}, authors generalize these results to operators with homogeneous kernel which we call "\textit{Hardy-Hilbert type}". 

Inequalities originated from \cite[Theorem 319]{HLP} with new settings have their own interests. The results listed in this section which are related to the proof of Theorem \ref{thmDYH} are Section 3.1 and Theorem \ref{thmWE}.

\subsection{Hardy-Littlewood inequality}
In what follows we will need the following lemma.
\begin{lemma}[Hardy-Littlewood]\label{leHL}
For all $\mu$ measurable function $f,g$, we have
\begin{equation*}
    \left|\int_{X}fgd\mu \right| \le \int_0^{\infty}f^*(t)g^*(t)dt
\end{equation*}
where $f^*$ and $g^*$ represent the decreasing rearrangement functions. 
\end{lemma}
The proof can be found in \cite[Theorem 2.4]{Maly}. Note that Hardy-Littlewood inequality connects the integral-type operator with the Lorentz quasi-norm since the right-hand side can be estimated as
\begin{equation*}
    \int_0^{\infty}f^*(t)g^*(t)dt = \int_0^{\infty} \left(t^{1/p}f^*(t)\right) \left(t^{1/p'}g^*(t)\right)\frac{dt}{t} \le \|f\|_{(p,q)} \|g\|_{(p',q')}
\end{equation*}
which plays a crucial role in the proof of our main result. We also mention the \textit{nested} property of Lorentz spaces. That is $\forall 0<p\le \infty$ and $0<q\le r\le \infty$ we have
\begin{equation*}
    \|f\|_{(p,r)} \lesssim \|f\|_{(p,q)}
\end{equation*}

Before stating the next theorem, we first recall that a linear operator on the measure space $(X,\mu)$ is of \textit{restricted weak type} $(p,q)$ if for all measurable subset $A$ of $X$ with finite measure we have
\begin{equation*}
    \|T(\chi_A)\|_{(q,\infty)} \lesssim \mu(A)^{1/p}
\end{equation*}
We will use the following interpolation theorem \cite[Theorem 1.4.19]{8}.

\begin{theorem}\label{thmLSI}
Let $0<r< \infty$, $0<p_0\ne p_1< \infty$, and $0<q_0\ne q_1\le \infty$ and let $(X,\mu)$, $(Y,\nu)$ be $\sigma \text{-}$ finite measure spaces. Let $T$ be a linear operator defined on the space of simple functions on $X$ and taking values in the set of measurable functions on $Y$. Assume that $T$ is of \textit{restricted weak type} $(p_0,q_0)$ and $(p_1,q_1)$. Then for $0<\theta<1$ and
\begin{equation*}
    \frac{1}{p} = \frac{1-\theta}{p_0} + \frac{\theta}{p_1}\quad \textrm{and} \quad \frac{1}{q} = \frac{1-\theta}{q_0} + \frac{\theta}{q_1}
\end{equation*}
we have
\begin{equation*}
    \|T(f)\|_{(q,r)} \lesssim \|f\|_{(p,r)}\quad \forall f\in L^{p,r}
\end{equation*}

\end{theorem}
For the proof we refer readers to \cite[Theorem 1.4.19]{8}. Note that the above result can be improved to strong type estimate, i.e. $\|T(f)\|_q \lesssim \|f\|_p$, only when $q\ge p$ by \textit{nested} property.

\begin{remark}\label{reinterpolation}
We can regain the boundedness of $R$ for $1<p<n^*$ by the following procedure. First by Theorem \ref{thmHA} we know that $R$ is weak type $(1,1)$ which directly implies $R$ is restricted weak $(1,1)$. By Theorem \ref{thmDYH} and \textit{nested} property of Lorentz spaces, we know that 
\begin{equation*}
    \|R(f)\|_{(n^*,\infty)}\lesssim \|R(f)\|_{(n^*,1)} \lesssim \|f\|_{(n^*,1)}
\end{equation*}
which implies $R$ is restricted weak $(n^*,n^*)$ since for any set $A$ with $\mu(A)<\infty$ we have
\begin{equation*}
    \|\chi_A\|_{(n^*,1)} = \int_0^\infty t^{1/n^*}\chi_A^*(t)\frac{dt}{t} = \int_0^\infty t^{1/n^* -1}\chi_{[0,\mu(A)]}(t)dt = n^*\mu(A)^{1/n^*}
\end{equation*}

Now an application of Theorem \ref{thmLSI} with $r=p$ for each $1<p<n^*$, we regain the boundedness of Riesz transform.
\end{remark}

\subsection{Hardy-Hilbert type inequalities on homogeneous spaces}

In \cite[Theorem 319]{HLP}, authors consider integral operators with homogeneous kernels on space $\mathbb{R}^+$. Before we state their original result, we recall that a kernel $K(x,y)$ is homogeneous of degree $\delta \in \mathbb{R}$ if 
\begin{equation*}
    K(\lambda x,\lambda y) = \lambda^{\delta} K(x,y) \quad \forall \lambda>0, \quad \forall x,y
\end{equation*}

\begin{theorem}[Hardy-Littlewood-P\'olya]\label{thmHLP1}
Suppose $1<p<\infty$ and $K(x,y)$ is non-negative defined on $\mathbb{R}^+ \times \mathbb{R}^+$ with homogeneous degree $-1$ such that
\begin{equation*}
    \int_0^\infty K(x,1)x^{-1/p}dx = \int_0^\infty K(1,y)y^{-1/p'}dy = k
\end{equation*}
then we have
\begin{equation*}
    \int_0^\infty \int_0^\infty K(x,y)f(x)g(y)dxdy \le k\left(\int_0^\infty |f(x)|^pdx\right)^{1/p} \left(\int_0^\infty |g(y)|^{p'}dy\right)^{1/p'}
\end{equation*}
\end{theorem}
Particularly, when $K(x,y) = \frac{1}{x+y}$, the theorem proves the $L^p$ boundedness of the Hilbert transform on half line. Moreover, for $K(x,y) = \frac{1}{x} \chi_{y\le x}(x)$, Theorem \ref{thmHLP1} is the famous Hardy's inequality. In the following, we call an operator "\textit{Hardy-Hilbert type}" if it is defined via a homogeneous kernel.  
\bigskip

For $d_1,d_2\ge 1$, we define measures $d\mu_j = r^{d_j-1}dr$ for $j=1,2$. And we consider multiplicative group $(\mathbb{R}^+,\frac{dx}{x})$. For simplicity, we still use $*$ to denote the convolution in this group i.e.
\begin{equation*}
    (f*g)(x):= \int_0^\infty f(y)g(x/y)\frac{dy}{y} = \int_0^\infty f(x/y)g(y)\frac{dy}{y}
\end{equation*}
To generalize Theorem \ref{thmHLP1}, we will use the following result. See \cite[Page from 18]{8} for more detail.
\begin{theorem}[Young's convolution inequality]\label{thmYoung}
Let $1\le p,q,r\le \infty$ such that
\begin{equation*}
    1+\frac{1}{q} = \frac{1}{p} + \frac{1}{r}
\end{equation*}
then for all $f\in L^p(\mathbb{R}^+,dx/x)$ and $g\in L^r(\mathbb{R}^+,dx/x)$ we have
\begin{equation*}
    \|f*g\|_{L^q(\mathbb{R}^+,dx/x)} \le \|g\|_{L^r(\mathbb{R}^+,dx/x)} \|f\|_{L^p(\mathbb{R}^+,dx/x)}
\end{equation*}
\end{theorem}

With Theorem \ref{thmYoung} in mind, we generalize Theorem \ref{thmHLP1} as follows.

\begin{theorem}\label{thmHLP2}
Let $1<p<\infty$, $K(x,y)$ is non-negative with homogeneous degree of $-n$ where $n=d_2/p + d_1/{p'}$ and
\begin{equation*}
    \int_0^\infty K(x,1)x^{d_2/p -1}dx = \int_0^\infty K(1,y)y^{d_1/{p'}-1}dy = k<\infty
\end{equation*}
Then the operator
\begin{equation*}
    \mathcal{K}(f)(x)= \int_0^\infty K(x,y)f(y)d\mu_1(y)
\end{equation*}
is bounded from $L^p(\mathbb{R}^+,d\mu_1)$ to $L^p(\mathbb{R}^+,d\mu_2)$ with norm at most $k$.

\end{theorem}

\begin{proof}
Let $g(x) = x^{d_1/p}f(x)$ and $h(x) = x^{n-d_1/{p'}}K(x,1)$. Then the operator applies to function $f$ is
\begin{gather*}
    \int_0^\infty K(x,y) f(y) y^{d_1} \frac{dy}{y} = \int_0^\infty y^{d_1/{p'}-n}K(x/y,1) \left(y^{d_1/p} f(y)\right) \frac{dy}{y}\\
    = x^{d_1/{p'}-n} \int_0^\infty h(x/y) g(y) \frac{dy}{y} = x^{d_1/{p'}-n} (h*g)(x)
\end{gather*}
Thus we have by Theorem \ref{thmYoung},
\begin{gather*}
    \|\mathcal{K}(f)\|_{L^p(\mathbb{R}^+,d\mu_2)} = \left(\int_0^\infty x^{p(d_1/{p'}-n)+d_2}|(h*g)(x)|^p \frac{dx}{x}\right)^{1/p}\\
    = \|h*g\|_{L^p(\mathbb{R}^+,\frac{dx}{x})} \le \|h\|_{L^1(\mathbb{R}^+,\frac{dx}{x})} \|g\|_{L^p(\mathbb{R}^+,\frac{dx}{x})} = \|h\|_{L^1(\mathbb{R}^+,\frac{dx}{x})} \|f\|_{L^p(\mathbb{R}^+,d\mu_1)}
\end{gather*}
Note that
\begin{equation*}
    \|h\|_{L^1(\mathbb{R}^+,\frac{dx}{x})} = \int_0^\infty x^{n-d_1/{p'}-1}K(x,1)dx = k
\end{equation*}
which completes the proof.

\end{proof}

Immediately, we have the following corollary.

\begin{corollary}\label{corHLP3}
Let $1<p\le q$ and $1+1/q = 1/p + 1/r$. $K(x,y)$ is non-negative with homogeneous degree of $-n$ where $n = d_1/{p'} + d_2/q$ and
\begin{equation*}
    \int_0^\infty K(x,1)^r x^{\frac{rd_2}{q}-1}dx = \int_0^\infty K(1,y)^r y^{\frac{rd_1}{p'}-1}dy = k
\end{equation*}
Then
\begin{equation*}
    \|\mathcal{K}(f)\|_{L^q(\mathbb{R}^+,d\mu_2)} \le k^{1/r} \|f\|_{L^p(\mathbb{R}^+,d\mu_1)}
\end{equation*}
\end{corollary}

Also, we can extend the above results to higher dimensions.

\begin{theorem}\label{thmHLPG}
Let $1<p\le q$ and $1+1/q = 1/p + 1/r$. $K(x,y)$ is a kernel defined on $\mathbb{R}^{d_2}\times \mathbb{R}^{d_1} \setminus (0,0)$ which is radial and non-negative with homogeneous degree of $-n$ where $n = d_1/p' + d_2/q$. Set $K_0(|x|,|y|) = K(x,y)$ be the function defined on $\mathbb{R}^+ \times \mathbb{R}^+$ and
\begin{equation*}
    \int_0^\infty K_0(s,1)^r s^{\frac{rd_2}{q}-1}ds = \int_0^\infty K_0(1,t)^r t^{\frac{rd_1}{p'}-1}dt = k
\end{equation*}
Then
\begin{equation*}
    \|\mathcal{K}(f)\|_{L^q(\mathbb{R}^{d_2})} \lesssim  \|f\|_{L^p(\mathbb{R}^{d_1})}
\end{equation*}
where 
\begin{equation*}
    \mathcal{K}(f)(x):= \int_{\mathbb{R}^{d_1}} K(x,y)f(y)dy
\end{equation*}

\end{theorem}

\begin{proof}
Note that
\begin{gather*}
    \mathcal{K}(f)(x) = \int_{\mathbb{R}^{d_1}}K(x,y)f(y)dy = C_{d_1} \int_{\mathbb{S}^{d_1-1}} \int_0^\infty K_0(|x|,t)f(\theta t) t^{d_1-1}dt d\theta\\
    := C_{d_1} \int_{\mathbb{S}^{d_1-1}} \mathcal{K}_\theta (f)(x) d\theta
\end{gather*}
Since $K$ is radial, thus for each $\theta \in \mathbb{S}^{d_1-1}$, $\mathcal{K}_\theta (f)$ is radial. Hence
\begin{gather*}
    \|\mathcal{K}_\theta (f)\|_{L^q(\mathbb{R}^{d_2})} = \left(\int_{\mathbb{R}^{d_2}}|\mathcal{K}_\theta (f)(x)|^q dx \right)^{1/q}\\
    = C_{q,d_2} \left(\int_0^\infty |\mathcal{K}_\theta (f)(|s|)|^q s^{d_2-1}ds\right)^{1/q} = C_{q,d_2} \|\mathcal{K}_\theta (f)(|.|)\|_{L^q(\mathbb{R}^+,d\mu_2)}
\end{gather*}
By Corollary \ref{corHLP3}, we have
\begin{equation*}
    \|\mathcal{K}_\theta (f)(|.|)\|_{L^q(\mathbb{R}^+,d\mu_2)} \le k^{1/r} \|f(\theta (.))\|_{L^p(\mathbb{R}^+,d\mu_1)}
\end{equation*}
Therefore 
\begin{gather*}
    \|\mathcal{K}(f)\|_{L^q(\mathbb{R}^{d_2})} \approx \|\int_{\mathbb{S}^{d_1-1}}\mathcal{K}_\theta (f)(x) d\theta \|_{L^q(\mathbb{R}^{d_2})} \le \int_{\mathbb{S}^{d_1-1}} \|\mathcal{K}_\theta (f)\|_{L^q(\mathbb{R}^{d_2},dx)} d\theta\\
    \lesssim \int_{\mathbb{S}^{d_1-1}} \|\mathcal{K}_\theta (f)(|.|)\|_{L^q(\mathbb{R}^+,d\mu_2(x))}d\theta \lesssim \int_{\mathbb{S}^{d_1-1}} \left(\int_0^\infty |f(\theta t)|^p t^{d_1-1}dt\right)^{1/p}d\theta\\
    \lesssim \left(\int_{\mathbb{S}^{d_1-1}} \int_0^\infty |f(\theta t)|^p t^{d_1-1}dtd\theta \right)^{1/p} \approx \|f\|_{L^p(\mathbb{R}^{d_1})}
\end{gather*}

\end{proof}

\subsection{Hardy-Hilbert type inequalities on inhomogeneous spaces}
To prove our main results Theorem \ref{thmDYH}, we need an "endpoint-type" version of Theorem \ref{thmHLP2} in some different settings. The initial point is to consider "\textit{Hardy-Hilbert type}" operator $\mathcal{K}$ acting between measure spaces $(X,\mu_1)$ and $(X,\mu_2)$ with kernel: 
\begin{equation*}
    K(x,y) = 
    \begin{cases}
       x^{-\alpha}y^{-\beta}, & x\le y\\
       x^{-\alpha'}y^{-\beta'}, & x> y
    \end{cases}
\end{equation*}
where $\alpha+\beta = \alpha'+\beta'$ satisfying $\alpha>0, 0<\beta\le d_1$ and $0<\alpha'\le d_2, \beta'>0$. Let $X=[1,\infty)$, $d\mu_j$ as defined in the last subsection and we want to study the $(p,q)$ boundedness of $\mathcal{K}$. That is we aim to find $(p,q)$ such that $\mathcal{K}$ maps $L^p(X,d\mu_1)$ to $L^q(X,d\mu_2)$.

Split $\mathcal{K}$ into two parts,
\begin{gather*}
    R_1(f)(x) = x^{-\alpha} \int_x^\infty y^{-\beta}f(y)d\mu_1(y) \\
    R_2(f)(x) = x^{-\alpha'}\int_1^x y^{-\beta'}f(y)d\mu_1(y) 
\end{gather*}
Observe that the adjoint operator of $R_1$:
\begin{equation*}
    R_1^*(f)(x) = x^{-\beta}\int_1^x y^{-\alpha}f(y)d\mu_2(y)
\end{equation*}
is exactly $R_2$ by replacing $\alpha \to \beta'$, $\beta \to \alpha'$, $d_1\to d_2$ and $d_2\to d_1$. Thus instead of studying $\mathcal{K}$ itself, it is enough to study the boundedness of $R_1$. And the results for $R_2$ follows by duality and replacing indexes. $\mathcal{K}$ has also been studied in \cite[Proposition 5.1]{GH} and \cite[Lemma 5.4]{13}. We mention that the results in the following do not directly imply Theorem \ref{thmDYH}, but the ideas are exactly the same.

\begin{theorem}\label{thmWE}
 For $0<\beta<d_1$, $R_1$ is bounded from $L^{p,1}(X,d\mu_1)$ to $L^{q,\infty}(X,d\mu_2)$ for $p = \frac{d_1}{d_1-\beta}$ and $q\ge \frac{d_2}{\alpha}$. Moreover, $R_1$ is of weak type $(1,r)$ for all $r\ge \frac{d_2}{\alpha+\beta}$ for all $0<\beta \le d_1$.
\end{theorem}

\begin{proof}
 Let $x\ge 1$, and $f\in L^{p,1}(X,d\mu_1)$. By Lemma \ref{leHL},
\begin{gather*}
    |R_1(f)(x)|\le x^{-\alpha} \int_1^\infty |f(y)| y^{-\beta}\chi_{[x,\infty)}(y) d\mu_1(y)\\
    \le x^{-\alpha} \|f\|_{L^{p,1}(d\mu_1)} \|F_x\|_{L^{p',\infty}(d\mu_1)}
\end{gather*}
where $F_x(y) = y^{-\beta}\chi_{[x,\infty)}(y)$. Note that
\begin{equation*}
    \|F_x\|_{L^{p',\infty}(d\mu_1)} = \sup_{\lambda>0}  \lambda d_{F_x}(\lambda)^{1/{p'}}
\end{equation*}
where $d_{F_x}(\lambda) = \mu_1\left(\{y\ge x: y^{-\beta}>\lambda\}\right)$.

Since
\begin{equation*}
    d_{F_x}(\lambda) = \frac{1}{d_1}(\lambda^{-d_1/\beta} - x^{d_1}) \chi_{0<\lambda<x^{-\beta}}
\end{equation*}
thus 
\begin{gather*}
    \|F_x\|_{L^{p',\infty}(d\mu_1)} = \left(\sup_{0<\lambda<x^{-\beta}} \lambda^{p'}d_1^{-1}(\lambda^{-{d_1}/\beta}-x^{d_1}) \right)^{1/{p'}}\\
    = d_1^{-1/{p'}} \left(\sup_{0<\lambda<x^{-\beta}} \lambda^{p'-d_1/\beta}- x^{d_1} \lambda^{p'}\right)^{1/{p'}} \le C_{d_1,p}
\end{gather*}
since $p' = d_1/\beta$. Note that this estimate is uniform in $x$. Thus, we get
\begin{equation}\label{eq333}
    |R_1(f)(x)|\le C_{d_1,p} \|f\|_{L^{p,1}(d\mu_1)} x^{-\alpha}
\end{equation}
Finally, since $x^{-\alpha} \in L^{q,\infty}(X,d\mu_2)$ whenever $q\ge d_2/\alpha$ the result of the first part follows.

For the second part, we simply have
\begin{equation}\label{2.14e}
    |R_1(f)(x)|\le x^{-\alpha}\int_x^\infty |f(y)|y^{-\beta}d\mu_1(y)\le x^{-(\alpha+\beta)}\|f\|_{L^1(d\mu_1)}
\end{equation}
and the result follows by $x^{-(\alpha+\beta)} \in L^{r,\infty}(X,d\mu_2)$ whenever $r\ge \frac{d_2}{\alpha+\beta}$.
\end{proof}

\begin{remark}\label{reRWE}
It is clear from \eqref{eq333}, 
\begin{equation*}
    \|R_1(f)\|_{L^q(X,d\mu_2)} \lesssim \|f\|_{L^{p,1}(X,d\mu_1)} \quad \forall q>d_2/\alpha
\end{equation*}
\end{remark}

Then by Theorem \ref{thmLSI}, we have $R_1$ is $(p,q)$ bounded in the intersection of the open quadrilateral with vertexes $(\frac{d_1-\beta}{d_1},0)$, $(1,0)$, $(1,\frac{\alpha+\beta}{d_2})$, $(\frac{d_1-\beta}{d_1}, \frac{\alpha}{d_2})$ and the set $\{(1/p,1/q)\in [0,1]\times[0,1]: 1/q\le 1/p\}$ in the plane $1/p\text{-}1/q$. Moreover, by Marcinkiewicz's interpolation theorem, the second part tells us $R_1$ is bounded from $L^1(X,d\mu_1)\to L^q(X,d\mu_2)$ for all $\frac{d_2}{\alpha+\beta}< q<\infty$. 

Next we treat the upper boundary of the quadrilateral. Let $\epsilon>0$ and set $\frac{1}{p_\epsilon} = \frac{d_1-\beta}{d_1}+\epsilon$ and $\frac{1}{q_\epsilon} = \frac{d_1}{d_2}\cdot \frac{1}{p_\epsilon} - \frac{d_1-\alpha-\beta}{d_2}$. Then the following lemma gives a weak estimate on that line segment.

\begin{lemma}
For $0<\beta \le d_1$, $R_1$ is of weak type $(p_\epsilon,q_\epsilon)$ for all $\epsilon>0$ such that $(1/p_\epsilon,1/q_\epsilon)\in (0,1)\times(0,1)$.
\end{lemma}

\begin{proof}
 Note that $\forall x\ge 1$, we have
 \begin{gather*}
     |R_1(f)(x)|\le x^{-\alpha} \int_x^\infty |f(y)|y^{-\beta}y^{d_1-1}dy\\
     \le x^{-\alpha}\|f\|_{p_\epsilon} \left(\int_x^\infty y^{d_1-\beta p_\epsilon' -1}\right)^{1/p_\epsilon' }\\
     \lesssim x^{d_1/{p_\epsilon'}-\alpha-\beta}\|f\|_{L^{p_\epsilon}(X,d\mu_1)}
 \end{gather*}
 
 Now suppose that for some $\lambda>0$ we have
 \begin{equation*}
     |R_1(f)(x)|>\lambda
 \end{equation*}
 then we have
 \begin{equation*}
     \lambda \lesssim x^{d_1/{p_\epsilon'}-\alpha-\beta}\|f\|_{p_\epsilon}
 \end{equation*}
 which implies
 \begin{equation*}
     x^{q_\epsilon(\alpha+\beta-d_1/{p_\epsilon'})} \lesssim \frac{\|f\|_{p_\epsilon}^{q_\epsilon}}{\lambda^{q_\epsilon}} \implies x\lesssim \left(\frac{\|f\|_{p_\epsilon}^{q_\epsilon}}{\lambda^{q_\epsilon}}\right)^{\frac{1}{q_\epsilon(\alpha+\beta-d_1/{p_\epsilon'})}}
 \end{equation*}
 thus 
 \begin{equation*}
     \mu_2\{x\ge 1:|R_1(f)(x)|>\lambda\} \lesssim \left(\frac{\|f\|_{p_\epsilon}^{q_\epsilon}}{\lambda^{q_\epsilon}}\right)^{\frac{d_2}{q_\epsilon(\alpha+\beta-d_1/{p_\epsilon'})}} = \left(\frac{\|f\|_{p_\epsilon}}{\lambda}\right)^{q_\epsilon}
 \end{equation*}
\end{proof}
\begin{remark}
Note that for the case $\beta=d_1$, $R_1$ is $(p,q)$ bounded for $1<p<\infty$ and $1/q< d_1/(d_2p)+\alpha/d_2$. Indeed, by Hölder's inequality,
\begin{equation*}
    |R_1(f)(x)|\le x^{-\alpha}\|f\|_{L^p(d\mu_1)} \left(\int_x^\infty y^{-d_1p' +d_1-1}dy\right)^{1/p'} \approx x^{\frac{d_1(1-p')}{p'}-\alpha}\|f\|_{L^p(d\mu_1)}
\end{equation*}
and $\|R_1(f)\|_q \lesssim \|f\|_p$ follows.
\end{remark}

To verify the negative part of Theorem \ref{thmWE}, we need the following lemma.

\begin{lemma}
Let $0< \beta \le d_1$.

$(a).$ For $\frac{1}{p} \le \frac{d_1-\beta}{d_1}$, $R_1$ is not bounded from $L^p(X,d\mu_1)$ to $L^q(X,d\mu_2)$ for any $1\le q\le \infty$. 

$(b).$ For $\frac{d_1-\beta}{d_1}<\frac{1}{p} \le 1$, $R_1$ is not bounded from $L^p(X,d\mu_1)$ to $L^q(X,d\mu_2)$ for
\begin{equation*}
    \frac{1}{q}>\frac{d_1}{d_2}\cdot \frac{1}{p} - \frac{d_1-\alpha-\beta}{d_2}
\end{equation*}
\end{lemma}

\begin{proof}
(a). Let $p\ge \min( \frac{d_1}{d_1-\beta},\infty)$. We consider function $f(y) = y^{\beta-d_1}(1+\log y)^{-1}$. We can see that if $\beta<d_1$,
\begin{gather*}
    \|f\|_p^p = \int_1^\infty y^{-p(d_1-\beta)}(1+ \log y)^{-p}y^{d_1-1}dy\\
    = \int_0^\infty e^{-cx}(1+x)^{-p}dx <\infty
\end{gather*}
where $c = p(d_1-\beta)-d_1\ge 0$.

And if $\beta = d_1$,
\begin{equation*}
    \|f\|_\infty \le 1
\end{equation*}
However $\forall x\in X$,
\begin{equation*}
    R_1(f)(x) = x^{-\alpha}\int_x^\infty y^{-1}(1+\log y)^{-1}dy = x^{-\alpha}\int_{\log x}^\infty (1+t)^{-1}dt = \infty
\end{equation*}
which gives a counter example.

(b). Let $1\le p<\min( \frac{d_1}{d_1-\beta}, \infty)$. We consider function $f(y) = y^{-\frac{d_1+\delta}{p}}$ for some $\delta>0$ to be determined later. Then
\begin{equation*}
    \|f\|_p^p = \int_1^\infty y^{-d_1-\delta+d_1-1}dy \lesssim_{\delta} 1
\end{equation*}
and
\begin{equation*}
    R_1(f)(x) = x^{-\alpha}\int_x^\infty y^{-\beta-\frac{d_1+\delta}{p}+d_1-1}dy \approx_{\beta,d_1,\delta,p} x^{d_1-\alpha-\beta-\frac{d_1+\delta}{p}}
\end{equation*}
since $p<\frac{d_1}{d_1-\beta}<\frac{d_1+\delta}{d_1-\beta}$.

Hence
\begin{equation}\label{2.15e}
    \|R_1(f)\|_q^q \approx \int_1^\infty x^{q(d_1-\alpha-\beta)-\frac{q}{p}(d_1+\delta)+d_2-1}dx
\end{equation}
Note that
\begin{equation*}
     \frac{1}{q}>\frac{d_1}{d_2}\cdot \frac{1}{p} - \frac{d_1-\alpha-\beta}{d_2}\implies 0<\frac{p}{q}\cdot d_2 - d_1 + p(d_1-\alpha-\beta)
\end{equation*}
Therefore, by choosing
\begin{equation*}
    0<\delta \le \frac{p}{q}\cdot d_2 - d_1 + p(d_1-\alpha-\beta)
\end{equation*}
\eqref{2.15e} is unbounded.

\end{proof}

Finally, we introduce the following lemma to complete the picture of the $(p,q)$ boundedness of $R_1$.

\begin{lemma}
Let $0<\beta \le d_1$.

(a) If $\alpha+\beta \le d_2$. Then $R_1$ is bounded from $L^1(X,d\mu_1)$ to $L^q(X,d\mu_2)$ where $q = \frac{d_2}{\alpha+\beta}$.

(b) For $\max (\frac{d_1-\beta}{d_1},0) < \frac{1}{p} \le 1$, $R_1$ is bounded from $L^p(X,d\mu_1)$ to $L^{\infty}(X,d\mu_2)$.

\end{lemma}

\begin{proof}
(a) Note that
\begin{gather*}
    R_1(f)(x) = x^{-\alpha}\int_x^\infty y^{-\beta}f(y)d\mu_1(y)\\
    \le \sum_{j\ge 0} 2^{-j\beta} \int_{2^j}^{2^{j+1}} \big(x^{-\alpha}\chi_{(1,y)}(x)\big) |f(y)|d\mu_1(y)
\end{gather*}
Then by Minkowski's integral inequality, we have
\begin{equation*}
    \|R_1(f)\|_{L^{q}(X,d\mu_2)}\le \sum_{j\ge 0} 2^{-j\beta} \int_{2^j}^{2^{j+1}} \|x^{-\alpha}\chi_{(1,y)}(x)\|_{L^{q}(X,d\mu_2)} |f(y)|d\mu_1(y)
\end{equation*}
Notice that
\begin{equation*}
    \|x^{-\alpha}\chi_{(0,y)}(x)\|_{L^{q}(X,d\mu_2)}^{q} = \int_1^y x^{-\frac{\alpha d_2}{\alpha+\beta}+d_2-1}dx \lesssim_{d_2,\alpha,\beta} y^{\frac{\beta d_2}{\alpha+\beta}}
\end{equation*}
Therefore 
\begin{gather*}
    \|R_1(f)\|_{L^q(X,d\mu_2)} \lesssim \sum_{j\ge 0} 2^{-j\beta} \int_{2^j}^{2^{j+1}}y^\beta |f(y)|d\mu_1(y)\\
    \le 2^\beta \sum_{j\ge 0} \int_{2^j}^{2^{j+1}}|f(y)|d\mu_1(y) = 2^\beta \|f\|_{L^1(X,d\mu_1)}
\end{gather*}

(b) The case $p=1$ is obvious and for any $1< p< \min ( \frac{d_1}{d_1-\beta}, \infty)$, by \eqref{2.14e} we have $\forall x\ge 1$,
\begin{gather*}
    |R_1(f)(x)|\le x^{-\alpha} \int_x^\infty |f(y)|y^{-\beta}d\mu_1(y)\\
    \le \|f\|_p \left(\int_x^{\infty} y^{d_1-\beta p'-1}\right)^{1/p'}\lesssim \|f\|_p 
\end{gather*}

\end{proof}

Introduce notions:
\begin{gather*}
    D_{d_1,d_2}^{\alpha,\beta} = \left\{\left(\frac{1}{p},\frac{1}{q}\right)\in [0,1]\times [0,1]: \frac{1}{q} \le \frac{d_1}{d_2}\cdot \frac{1}{p} - \frac{d_1-\alpha-\beta}{d_2}, \frac{d_1-\beta}{d_1}<\frac{1}{p}\le1\right\}\\
    F_{d_1,d_2}^{\alpha,\beta} = \left\{\left(\frac{1}{p},\frac{1}{q}\right)\in [0,1]\times [0,1]: \frac{1}{q} \le \frac{d_1}{d_2}\cdot \frac{1}{p} - \frac{d_1-\alpha-\beta}{d_2}, \frac{d_1-\alpha- \beta}{d_1} \le \frac{1}{p}< \frac{d_1-\beta}{d_1}\right\} \bigcup\\
    \left\{\left(\frac{1}{p},\frac{1}{q}\right)\in [0,1]\times [0,1]: \frac{1}{q} < \frac{\alpha}{d_2}, \frac{d_1-\beta}{d_1} \le \frac{1}{p}\le 1\right\}
\end{gather*}
Now putting things together, we have the following. 
\begin{theorem}\label{thmPQ}
 For $1\le p,q\le \infty$, $0<\beta\le d_1$ and $\alpha>0$, $R_1$ is $(p,q)$ bounded if and only if $(1/p,1/q)\in D_{d_1,d_2}^{\alpha,\beta}$.
\end{theorem}

\begin{corollary}\label{corPQ}
For $1\le p,q\le \infty$, $0<\alpha'\le d_2$ and $\beta' >0$, $R_2$ is $(p,q)$ bounded if and only if $(1/p,1/q)\in F_{d_1,d_2}^{\alpha',\beta'}$.
\end{corollary}

\begin{proof}
    By duality of Theorem \ref{thmPQ}.
\end{proof}

\subsection{The strong $(p,q)$ estimates for $G_3$ and $G_4$}
As discussed in section 1, we can decompose the low energy part of Riesz transform as
\begin{gather*}
    R_L = \frac{2}{\pi}\nabla \int_0^{1}(\Delta+k^2)^{-1}dk = \frac{2}{\pi}\nabla \int_0^{1}\left(G_1(k) + G_2(k) + G_3(k) + G_4(k) \right) dk
\end{gather*}
It is convenient to define
\begin{equation*}
    T_j = \frac{2}{\pi} \nabla \int_0^1 G_j(k)dk\quad j=1,2,3,4
\end{equation*}
In this subsection, we focus on $T_3$ and $T_4$ and obtain the $p\text{-}q$ estimates for them. For each $1\le i,j\le l$, let $z_i^0 \in K_i$, $z_j^0\in K_j$ be fixed. Set $d(z_i^0, K_i^c) = d(z_j^0,K_j^c)=1$, $r = d(z_i^0,z)$, $r' = d(z_j^0,z')$, $E_i = \mathbb{R}^{n_i}\times \mathcal{M}_i \setminus K_i$ for simplicity and use estimates \eqref{2.5e}, \eqref{2.8e} to get the bound of the kernel of $T_3$ i.e. $\int_0^1\nabla G_3(k)dk$ as follows
\begin{equation}
    \begin{cases}
    1, & z\in K, z'\in K\\
    (r')^{1-n_j}, & z\in K, z'\in E_j\\
    r^{-n_i}, & z\in E_i, z'\in K\\
    \min \left(r^{-n_i} (r')^{2-n_j} ; r^{1-n_i}(r')^{1-n_j}\right), & z\in E_i, z'\in E_j
    \end{cases}
    \nonumber
\end{equation}
and therefore, for each $1\le i\le l$,
\begin{equation}\label{eestimates}
T_3(f)(z)\lesssim
    \begin{cases}
    \sum_{j=1}^l \int_{E_j} (r')^{1-n_j} |f(z')|d\mu(z') + \int_K |f(z')|d\mu(z'), & z\in K\\
    \sum_{j=1}^l\int_{E_j}  \min \left(r^{-n_i} (r')^{2-n_j} ; r^{1-n_i}(r')^{1-n_j}\right) |f(z')|d\mu(z')+\\
    r^{-n_i} \int_K |f(z')|d\mu(z'), & z\in E_i 
    \end{cases}
\end{equation}
Similarly, the bound of the kernel of $T_4$ is like
\begin{equation}
    \begin{cases}
    1, & z\in K, z'\in K\\
    (r')^{-n_j}, & z\in K, z'\in E_j\\
    r^{-n_i}, & z\in E_i, z'\in K\\
     \min \left(r^{-n_i}(r')^{1-n_j} ; r^{1-n_i} (r')^{-n_j}\right), & z\in E_i, z'\in E_j
    \end{cases}
    \nonumber
\end{equation}
and therefore 
\begin{equation}\label{eqT4}
T_4(f)(z)\lesssim
    \begin{cases}
    \sum_{j=1}^l \int_{E_j} (r')^{-n_j} |f(z')|d\mu(z') + \int_K |f(z')|d\mu(z'), & z\in K\\
    \sum_{j=1}^l\int_{E_j}  \min \left(r^{-n_i} (r')^{1-n_j} ; r^{1-n_i}(r')^{-n_j}\right) |f(z')|d\mu(z')+\\
    r^{-n_i} \int_K |f(z')|d\mu(z'), & z\in E_i 
    \end{cases}
\end{equation}
We give the results in the following theorem.

\begin{theorem}\label{thmT3PQ}
Let $1\le p,q\le \infty$. $T_3$ is strong $(p,q)$ if $(1/p,1/q)\in \bigcap_{1\le i,j\le l} D_{n_j,n_i}^{n_i-1,n_j-1} \setminus \{q=1\}$. And $T_4$ is strong $(p,q)$ if $(1/p,1/q)\in \bigcap_{1\le i,j\le l} F_{n_j,n_i}^{n_i,n_j-2} \setminus \{p=\infty\}$

\end{theorem}

\begin{proof}
First by the above estimates \eqref{eestimates}, we have
\begin{gather}\label{e3.0}
    \|T_3(f)\|_{L^q(\mathcal{M})}^q \lesssim \int_K \left(\sum_{j=1}^l\int_{E_j} (r')^{1-n_j} |f(z')|d\mu(z') + \int_K |f(z')|d\mu(z') \right)^q d\mu(z) +\\ \label{e3.1}
    \sum_{i=1}^l\int_{E_i} \left( \sum_{j=1}^l\int_{E_j}  \min \left(r^{-n_i} (r')^{2-n_j} ; r^{1-n_i}(r')^{1-n_j}\right) |f(z')|d\mu(z')+  r^{-n_i} \int_K |f(z')|d\mu(z') \right)^qd\mu(z)
\end{gather}
Note that  
\begin{equation*}
    \int_{E_j}(r')^{p'(1-n_j)}d\mu(z') \approx \int_1^\infty (r')^{p'(1-n_j)+n_j-1}dr' <\infty
\end{equation*}
if $p'>\frac{n_j}{n_j-1}$ i.e. $1\le p<n_j$. Thus for all $1\le p<n^*$ the RHS of \eqref{e3.0} can be bounded by some constant multiple of
\begin{gather*}
    \sum_{j=1}^l \int_K \left( \|f\|_{L^p(\mathcal{M})}\right)^qd\mu(z)+ \int_K \left(\int_K |f(z')|d\mu(z')\right)^qd\mu(z)\\
    \lesssim \|f\|_{L^p(\mathcal{M})}^q \quad \forall q\ge 1
\end{gather*}
since $K$ is compact. 

Next we analyse \eqref{e3.1}. Note that
\begin{gather*}
    \eqref{e3.1} \lesssim \sum_{i,j=1}^l \int_{E_i} \left(\int_{E_j} \min \left(r^{-n_i} (r')^{2-n_j} ; r^{1-n_i}(r')^{1-n_j}\right) |f(z')|d\mu(z')\right)^q d\mu(z) +\\
    \sum_{i=1}^l \int_{E_i} r^{-qn_i} \left(\int_K |f(z')|d\mu(z')\right)^q d\mu(z)
\end{gather*}
and the second term can be simply bounded by 
\begin{equation*}
    \|f\|_{L^p(\mathcal{M})}^q \sum_{i=1}^l \int_1^\infty r^{n_i-qn_i-1}dr<\infty
\end{equation*}
as long as $q>1$ and for all $1\le p\le \infty$. 

Finally, for the first line, for each $1\le i,j\le l$, we bound it as
\begin{gather*}
    \int_{E_i} \left(\int_{E_j\cap \{r\le r'\}}r^{1-n_i}(r')^{1-n_j}|f(z')|d\mu(z')\right)^q d\mu(z)\\
    + \int_{E_i}\left(\int_{E_j\cap \{1\le r'\le r\}}r^{-n_i}(r')^{2-n_j}|f(z')|d\mu(z')\right)^q d\mu(z)\\
    := \int_{E_i} |S_1(f)(z)|^q d\mu(z) + \int_{E_i} |S_2(f)(z)|^q d\mu(z)
\end{gather*}
Note that $S_1(f)$ is radial and hence transfer to polar coordinate,
\begin{gather*}
    \|S_1(f)\|_{L^q(E_i)}^q \approx \int_1^\infty |S_1(f)(r)|^q r^{n_i-1}dr = \|S_1(f)\|_{L^q([1,\infty),r^{n_i-1}dr)}^q
\end{gather*}
and
\begin{gather*}
    S_1(f)(z) \approx \int_{S_j} \int_{r}^\infty r^{1-n_i} (r')^{1-n_j} |f(r',\theta')|(r')^{n_j-1}dr'd\theta' := \int_{S_j} S_1^{\theta'}(f)(r)d\theta'
\end{gather*}
where $S_j$ denotes the sphere of the unit geodesic ball of $E_j$.

Therefore by Minkowski's integral inequality and Theorem \ref{thmPQ} with $\alpha = n_i-1$, $\alpha'=n_i$, $\beta=n_j-1$, $\beta'=n_j-2$, $d_1=n_j$, $d_2=n_i$. We have
\begin{gather*}
     \|S_1(f)\|_{L^q(E_i)} \lesssim \int_{S_j} \|S_1^{\theta'}(f)\|_{L^q([1,\infty),r^{n_i-1}dr)}d\theta' \lesssim \int_{S_j} \|f(.,\theta')\|_{L^p([1,\infty),(r')^{n_j-1}dr)} d\theta'\\
     \lesssim \left(\int_{S_j}\int_1^\infty |f(r',\theta')|^p (r')^{n_j-1}dr'd\theta'\right)^{1/p}\lesssim \|f\|_{L^p(\mathcal{M})}
\end{gather*}
if $(p,q)\in D_{n_j,n_i}^{n_i-1,n_j-1}$. Similarly, we can estimate $S_2(f)$ in the same way but using Corollary \ref{corPQ}. Then we have $S_2$ is bounded from $L^p(\mathcal{M})\to L^q(E_i)$ if $(1/p,1/q)\in F_{n_j,n_i}^{n_i,n_j-2}$.

Notice that in our setting $n_i,n_j\ge 3$, $\alpha+\beta=n_i+n_j-2>n_j,n_i$. Hence the first line of \eqref{e3.1}, for each $i,j$, is strong $(p,q)$ if in the region:
\begin{gather*}
    D_{n_j,n_i}^{n_i-1,n_j-1} \cap F_{n_j,n_i}^{n_i,n_j-2} = D_{n_j,n_i}^{n_i-1,n_j-1} \setminus \{q=1\}
\end{gather*}
which can be seen in Figure 1.
\begin{figure}
  \begin{minipage}[t]{0.5\linewidth}
    \centering
        \begin{tikzpicture}[scale=4.0]
        \draw (0,0) -- (1,0) -- (1,1) -- (0,1) -- cycle;
        \filldraw[fill=black!25] (6/16,0)--(1,0)--(1,1)--(19/24,1)--(6/16,12/16)--(6/16,0) -- cycle;
        \draw (0,0) node[below] {$0$}    (6/16,0) node[below] {$\frac{1}{n_j}$} (1,0) node[below] {$1$} (1.1,0) node[below] {$1/p$} (0,12/16) node[left] {$\frac{n_i-1}{n_i}$} (19/24,1) node[above] {$(\frac{2}{n_j},1)$};
        \draw  (0,1) node[left] {$1$} (0,1.1) node[left] {$1/q$};
        \draw [dashed] (0,0) -- (1,1);
        \draw[thick,red] (6/16,0) -- (6/16,12/16);
        \draw[thick,blue] (6/16,12/16) -- (19/24,1);
        \draw[dashed] (19/24,1) -- (1,18/16);
        \draw[thick,red] (19/24,1)--(1,1);
        \draw[thick,blue] (1,1)--(1,0);
        \draw[thick,blue] (1,0)--(6/16,0);
        \draw[dashed] (0,12/16) -- (6/16,12/16);
        \draw[dashed] (1,1) -- (1,18/16);
        \draw (1,18/16) node[right] {$(1,\frac{n_i+n_j-2}{n_i})$};
        \draw (11/16,44/100) node {$D_{n_j,n_i}^{n_i-1,n_j-1}$};
        \end{tikzpicture}
        \caption{Boundedness for $T_3$}
        \label{a}
  \end{minipage}%
  \begin{minipage}[t]{0.5\linewidth}
    \centering
        \begin{tikzpicture}[scale=4.0]
        \draw (0,0) -- (1,0) -- (1,1) -- (0,1) -- cycle;
        \filldraw[fill=black!25] (0,0)--(1,0)--(1,1)--(19/24,1)--(0,12/16)--(0,0) -- cycle;
        \draw (0,0) node[below] {$0$}    (19/24,0) node[below] {$\frac{1}{n_j}$} (1,0) node[below] {$1$} (1.1,0) node[below] {$1/p$} (0,12/16) node[left] {$\frac{n_i-1}{n_i}$};
        \draw  (0,1) node[left] {$1$} (0,1.1) node[left] {$1/q$};
        \draw [dashed] (0,0) -- (1,1);
        \draw[thick,red] (0,0) -- (0,12/16);
        \draw[thick,blue] (0,12/16) -- (19/24,1);
        \draw[dashed] (19/24,1) -- (1,81/76);
        \draw[thick,red] (19/24,1)--(1,1);
        \draw[thick,blue] (1,1)--(1,0);
        \draw[thick,blue] (1,0)--(0,0);
        \draw[dashed] (19/24,0)--(19/24,1);
        \draw[dashed] (1,1) -- (1,81/76);
        \draw (1,81/76) node[right] {$(1,\frac{n_i+n_j-2}{n_i})$};
        \draw (7/16,66/100) node {$F_{n_j,n_i}^{n_i,n_j-2}$};
        \end{tikzpicture}
        \caption{Boundedness for $T_4$}
        \label{b}
  \end{minipage}
\end{figure}
And \eqref{e3.1} is strong $(p,q)$ in the shaded open region A and the blue solid lines except four end points of the red line. Finally, after taking intersection and recalling the discussion about \eqref{e3.0}, the result for $T_3$ follows. And the estimates for $T_4$ is similar but use estimates \eqref{eqT4}. The picture for $T_4$, each $i,j$, is shown in Figure 2.

\end{proof}

\section{Proof of theorem 1.3}
Recall the notion in the last section, we decompose $R_L$ as
\begin{gather*}
    R_L = \frac{2}{\pi}\nabla \int_0^{1}(\Delta+k^2)^{-1}dk = \frac{2}{\pi}\nabla \int_0^{1}\left(G_1(k) + G_2(k) + G_3(k) + G_4(k) \right) dk\\
    := T_1 + T_2 +T_3 +T_4
\end{gather*}

Referring to \cite[Proposition 5.1]{HaSi}, the operator $G_2(k)$ is a family of pseudodifferential operators with order $-2$ thus $\nabla G_2(k)$ are pseudodifferential operators with order $-1$. Then together with the compactness of $K$ we get that $T_2$ is bounded in $L^p$ for all $1\le p\le \infty$. The estimates of $T_4$ term is based on the $L^{p'}(\mathcal{M},d\mu(z');L^{p}(\mathcal{M},d\mu(z)))$ $(\forall 1<p<\infty)$ boundedness of its Schwartz kernel. In specific, according to $\eqref{2.13e}$ we can bound the kernel of $T_4$ by
\begin{equation*}
    O\big(d(z_i^0,z)^{-n_i} \big)
\end{equation*}
for $z'$ lies in some compact set and $z$ goes to infinity in $E_i$. And
\begin{equation*}
    O\big(d(z_j^0,z')^{-n_j}\big)
\end{equation*}
for $z$ lies in some compact set and $z'$ goes to the infinity in $E_j$ and
\begin{equation*}
    O \left( \min \left(d(z_i^0,z)^{-n_i} d(z_j^0,z')^{1-n_j}, d(z_i^0,z)^{1-n_i} d(z_j^0,z')^{-n_j} \right) \right)
\end{equation*}
see \cite[Proposition 4.1]{HaSi}. And then the $L^p$ $(1<p<\infty)$ boundedness of $T_4$ is followed by Minkowski's integral inequality and Hölder's inequality. 
\begin{gather*}
    \|\int_{\mathcal{M}} K(z,z')f(z')d\mu(z')\|_{L^p(\mathcal{M},d\mu(z))} \le \int_{\mathcal{M}} \|K(z,z')\|_{L^p(\mathcal{M},d\mu(z))}|f(z')|d\mu(z')\\
    \le \|K(z,z')\|_{L^{p'}(\mathcal{M},d\mu(z');L^p(\mathcal{M},d\mu(z)))} \|f\|_{L^p(\mathcal{M})}
\end{gather*}
where $K$ denotes the kernel of $T_4$.

Hence we only need to focus on $T_1$ and $T_3$ since the end point estimates for other terms are automatically hold by Theorem \ref{thmLSI} with $r=1$ i.e.
\begin{equation*}
    \|T_j(f)\|_{(n^*,1)} \le C\|f\|_{(n^*,1)} 
\end{equation*}
where $j = 2,4$. 

Therefore it is sufficient to prove the end point estimates for $T_1$ and $T_3$ terms. Which can be expressed as follows
\begin{equation} \label{3.1e}
    T_1(f)(z) = \frac{2}{\pi} \sum_{i=1}^{l} \int_{ \mathcal{M}}f(z') \phi_i(z') \int_0^{1} \nabla_{z}\left( (\Delta_{i}+k^2)^{-1}(z,z') \phi_i(z) \right) dk  d\mu(z')
\end{equation}

\begin{equation}\label{3.2e}
    T_3(f)(z) = \frac{2}{\pi} \sum_{i = 1}^{l} \int_{\mathcal{M}}f(z') \phi_i(z') \int_0^{1}(\Delta_i +k^2)^{-1}(z_i^0, z') \nabla_z(u_i(z,k)) dk d\mu(z')
\end{equation}
where we use $\Delta_i$ to denote $\Delta_{\mathbb{R}^{n_i}\times \mathcal{M}_i}$ for simplicity.

We consider them individually. First for $T_1$, we recall that the support of $\phi_i$ is in the end $E_i$. Therefore, instead of proving the end point estimates of $T_1$, it is sufficient to show the end point result for each term of the sum which is an operator defined on $E_i$ (for $1\le i \le l$)
\begin{equation*}
   \int_{E_i}f(z')\phi_i(z') \int_0^{1} \nabla_{z}\left( (\Delta_{i}+k^2)^{-1}(z,z') \phi_i(z) \right) dk  d\mu_i(z')
\end{equation*}
Next, we expand the gradient part as 
\begin{equation*}
    \nabla_{z}\left( (\Delta_{i}+k^2)^{-1}(z,z') \phi_i(z) \right) =\left( \nabla_{z} (\Delta_{i}+k^2)^{-1}(z,z')\right) \phi_i(z) + \nabla_{z}\left(\phi_i(z) \right) (\Delta_{i}+k^2)^{-1}(z,z')
\end{equation*}
Note that when the gradient hits the resolvent, the first term, its corresponding operator is relatively easy to handle. Since for $1<p<\infty$
\begin{gather}
    \|\phi_i \int_0^{1}\nabla (\Delta_i +k^2)^{-1}dk \phi_i \|_{p\to p} \lesssim \|\int_0^{1}\nabla (\Delta_i +k^2)^{-1}dk \|_{p\to p}\nonumber \\ \label{3.3e}
    = \|(\nabla \Delta_i^{-\frac{1}{2}}) \left(\frac{\pi}{2} - \rm{tan}^{-1} (\sqrt{\Delta_i})   \right)\|_{p\to p} \lesssim \|\nabla \Delta_i^{-\frac{1}{2}} \|_{p\to p}<\infty
\end{gather}
where the last inequality follows from standard results of Riesz transform and the second last inequality follows from \cite{7} or \cite[Lemma 2.2]{HaSi}.

Then we treat the operator where the gradient hits $\phi_i(z)$.

\begin{proposition}\label{prop1}
Let $P$ be the operator defined on the $ith$ end where
\begin{equation*}
    P(f)(z) = \nabla_z \phi_i(z) \int_{E_i}f(z')\phi_i(z') \int_0^{1}(\Delta_i +k^2)^{-1}(z,z')dk d\mu_i(z')
\end{equation*}
Then $P$ is bounded from $L^{n^* ,1}$ to $L^{n^*,1}$ where $n^*  = \min_{k}n_k$.
\end{proposition}

\begin{proof}
 Let $\epsilon >0$. Set $D = \{(z,z')\in \mathcal{M}^2 : d(z,z')\le \epsilon\}$. And let $\chi_{D}$ be its characteristic function. Then we write the Schwartz kernel as
 \begin{equation*}
     K(z,z') = \int_0^{1} (\Delta_i +k^2)^{-1}(z,z')dk = \int_0^{1} \chi_D (\Delta_i +k^2)^{-1}(z,z')dk + \int_0^{1} (1-\chi_D) (\Delta_i +k^2)^{-1}(z,z')dk
 \end{equation*}
Denote by $K_1$ the first term, we note that by \eqref{2.5e}
\begin{equation*}
    \sup\limits_{z} \int_{E_i}|K_1(z,z')|d\mu(z') \lesssim \sup\limits_{z} \int_{0}^{\epsilon}r^{1-N} r^{N-1}dr <\infty
\end{equation*}
where $r = d(z,z')$. Similarly, we have
\begin{equation*}
    \sup\limits_{z'} \int_{E_i}|K_1(z,z')|d\mu(z) \lesssim \sup\limits_{z'} \int_{0}^{\epsilon}r^{1-N} r^{N-1}dr <\infty
\end{equation*}
Therefore by Schur test, we have for all $1\le p\le \infty$
\begin{equation*}
    \|\int_0^{1} \chi_D (\Delta_i +k^2)^{-1}dk\|_{p\to p} <\infty
\end{equation*}
an application of Theorem \ref{thmLSI} with $r=1$ gives the endpoint result.

Now for the second term, we consider the operator
\begin{equation*}
    T(f)(z) = \nabla_z \phi_i(z) \int_{E_i}\phi_i(z')f(z')\int_0^{1}(1-\chi_D)(\Delta_i +k^2)^{-1}(z,z')dkd\mu_i(z')
\end{equation*}
By using \eqref{2.5e} and Lemma \ref{leHL}, we have $\forall z\in E_i $,
\begin{gather*}
    |T(f)(z)|\lesssim |\nabla_z \phi_i(z)| \int_{\{z':d(z,z')> \epsilon\}}|\phi_i(z')|  \int_0^{1}|(\Delta_i+k^2)^{-1}(z,z')|dk|f(z')|d\mu_i(z')\\
    \lesssim |\nabla_z \phi_i(z)|\int_{E_i} \left(d(z,z')^{1-n_i}(1-\chi_D) \right) |f(z')|d\mu_i(z')\\
    \lesssim |\nabla_z \phi_i(z)| \int_0^{\infty}g^*(t)f^*(t)dt = |\nabla_z \phi_i(z)| \int_0^{\infty} \left( t^{1/(n^* )'} g^*(t)\right) \left( t^{1/n^* } f^*(t)\right)\frac{dt}{t}\\
    \lesssim |\nabla_z \phi_i(z)| \|f\|_{(n^* ,1)} \sup\limits_{\alpha>0} \alpha d_g(\alpha)^{1/(n^* )'}
\end{gather*}
where $g(z') = d(z,z')^{1-n_i} (1-\chi_D)$.

Recall the notion in Lorentz space, $\forall 0<\alpha< \epsilon^{1-n_i}$ 
\begin{equation*}
    d_g(\alpha) = \mu_i\{z': d(z,z')>\epsilon, d(z,z')^{1-n_i}>\alpha\}\lesssim \alpha^{\frac{-n_i}{n_i-1}}
\end{equation*}
and $d_g(\alpha) = 0$ if $\alpha \ge \epsilon^{1-n_i}$.

Hence 
\begin{equation*}
    \sup\limits_{\alpha>0}\alpha d_g(\alpha)^{1/(n^* )'} \lesssim \sup\limits_{0<\alpha<\epsilon^{1-n_i}} \alpha^{1-\frac{n_i(n^* -1)}{(n_i-1)n^* }} = \epsilon^{1-n_i/n^* } 
\end{equation*}
Therefore, we have $\forall z\in E_i$,
\begin{equation*}
    |T(f)(z)|\lesssim  |\nabla_z \phi_i(z)| \|f\|_{(n^* ,1)}
\end{equation*}
Consequently
\begin{equation*}
    \|T(f)\|_{(n^*,1) } \lesssim \|f\|_{(n^* ,1)}
\end{equation*}
since $\nabla_z \phi_i$ is bounded and compactly supported:
\begin{equation*}
    \|\nabla \phi_i \|_{(n^*,1)} \approx \int_0^\infty d_{\nabla \phi_i}(\alpha)^{1/n^*} d\alpha \le  \mu_i\left(\textrm{supp}(\nabla \phi_i) \right)^{1/n^*} \|\nabla \phi_i\|_\infty<\infty
\end{equation*}
\end{proof}

Combining Proposition \ref{prop1} and \eqref{3.3e}, we have proved that $T_1$ is bounded from $L^{(n^* ,1)}$ to $L^{(n^*,1)}$. Next, we prove that $T_3$ also enjoys this property.

\begin{proposition}\label{prop2}
Let $T_3$ be the operator defined in \eqref{3.2e}. Then there exists some constant $C>0$ such that
\begin{equation*}
    \|T_3(f)\|_{(n^*,1) } \le C \|f\|_{(n^* ,1)}
\end{equation*}
\end{proposition}

\begin{proof}
 Let $z_i^0, z_j^0$ be fixed points in $K_i, K_j$ respectively. For simplicity we set 
 \begin{gather*}
     d(z_i^0 , K_i^{c}) = d(z_j^0 , K_j^{c}) =1
 \end{gather*}
Since $T_3$ is defined on whole $\mathcal{M}$. It suffices to investigate the behaviors of operator defined on $\mathcal{M}$
\begin{equation*}
    S(f)(z) = \int_{E_i}\phi_i(z') f(z')\int_0^{1}(\Delta_i +k^2)^{-1}(z_i^0,z') (\nabla_z u_i)(z,k)dk d\mu_i(z')
\end{equation*}
for $1\le i\le l$.

We first consider the case when $z$ is not in the connection $K$. Let $z\in E_j$ where $1\le i,j\le l$. Then by Lemma \ref{leKey} we have
\begin{equation*}
    S(f)(z) \lesssim \int_{E_i} |f(z')|\left(\int_0^{1}d(z_i^0, z')^{2-n_i} d(z_j^0,z)^{1-n_j}e^{-ck(d(z_i^0,z')+d(z_j^0,z))}dk\right)d\mu_i(z')
\end{equation*}
Set $r = d(z_j^0, z)$ and $r'= d(z_i^0, z')$, we have
\begin{gather*}
    |S(f)(z)|\lesssim \int_{E_i}|f(z')| \frac{(r')^{2-n_i} r^{1-n_j}}{r'+r}d\mu_i(z')\\
    \lesssim \int_{E_i}|f(z')| G(z,z')d\mu_i(z')
\end{gather*}
where $G(z,z') = \min\big((r')^{1-n_i}r^{1-n_j}, (r')^{2-n_i}r^{-n_j}\big)$
\begin{equation*}
     =  
    \begin{cases}
    (r')^{1-n_i}r^{1-n_j}, & r' > r\\
     (r')^{2-n_i}r^{-n_j}, & 1 \le r' \le  r
    \end{cases}
\end{equation*}
Therefore
\begin{gather*}
    |S(f)(z)|\lesssim r^{-n_j} \int_{1 \le r' \le  r}(r')^{2-n_i}|f(z')|d\mu_i(z') + r^{1-n_j} \int_{r'> r}(r')^{1-n_i}|f(z')|d\mu_i(z')
\end{gather*}
Now a similar argument in Theorem \ref{thmWE} gives that
\begin{equation*}
    |S(f)(z)|\lesssim \|f\|_{(n^*,1)} \left(r^{-n_j} \sup_{\alpha>0} \alpha d_{g_1}(\alpha)^{1/(n^*)'} + r^{1-n_j} \sup_{\alpha>0} \alpha d_{g_2}(\alpha)^{1/(n^*)'} \right)
\end{equation*}
where $g_1(z') = (r')^{2-n_i} \chi_{1 \le r' \le r}$ and $g_2(z') = (r')^{1-n_i} \chi_{r' > r}$.

Note that $\forall \alpha >0$
\begin{equation}\label{3.5e}
    d_{g_1}(\alpha) \lesssim  
    \begin{cases}
    0, & \alpha \ge 1\\
    \alpha^{\frac{-n_i}{n_i-2}}, & r^{2-n_i} \le  \alpha <1 \\
    r^{n_i}, & 0 < \alpha < r^{2-n_i}
    \end{cases} \quad
        d_{g_2}(\alpha) \lesssim   
    \begin{cases}
    0, & \alpha \ge r^{1-n_i}\\
     \alpha^{\frac{-n_i}{n_i-1}}, & 0<\alpha<r^{1-n_i}
    \end{cases}
\end{equation}
Therefore
\begin{gather*}
    \sup\limits_{\alpha>0}\alpha d_{g_1}(\alpha)^{1/(n^* )'}\lesssim r^{2-\frac{n_i}{n^* }}\quad
    \sup\limits_{\alpha>0}\alpha d_{g_2}(\alpha)^{1/(n^* )'}\lesssim r^{1-\frac{n_i}{n^* }}
\end{gather*}
Consequently, we have $\forall z\in E_j$, $(1\le j\le l)$
\begin{equation*}
    |S(f)(z)| \lesssim \|f\|_{(n^* ,1)} r^{2-n_j- n_i/n^* }
\end{equation*}
Finally if $z\in K$. We simply have
\begin{equation*}
    |S(f)(z)|\lesssim \int_{E_i} |f(z')| (r')^{1-n_i} d\mu_i(z') \lesssim \|f\|_{(n^* ,1)} \sup\limits_{\alpha>0} \alpha d_{g_3}(\alpha)^{1/(n^* )'}
\end{equation*}
where $g_3(z') = (r')^{1-n_i} \chi_{r'\ge  1}$.

A direct calculation gives that
\begin{equation*}
    \sup\limits_{\alpha>0}\alpha d_{g_3}(\alpha)^{1/(n^* )'} \lesssim 1^{-(n_i-n^*)/n^*} = 1
\end{equation*}
Whence 
\begin{gather*}
|S(f)(z)|\lesssim \|f\|_{(n^*,1)} \left(\chi_{K}(z) + \sum_{j=1}^l \chi_{E_j}(z)[d(z,z_j^0)]^{2-n_j-n_i/n^*}\right)
\end{gather*}
Since $K$ is compact we have 
\begin{gather*}
    \|S(f)\|_{(n^*,1)} \lesssim \|f\|_{(n^*,1)} \sum_{j=1}^l \int_0^\infty d_{\chi_{E_j}(z)[d(z,z_j^0)]^{2-n_j-n_i/n^*}}(\alpha)^{1/n^*} d\alpha\\
    \lesssim \|f\|_{(n^*,1)} \sum_{j=1}^l \int_0^1 \alpha^{\frac{n_j}{2n^*-n^* n_j-n_i}}d\alpha \lesssim \|f\|_{(n^*,1)}
\end{gather*}
since $\frac{n_j}{2n^*-n^* n_j-n_i}+1 = \frac{n_j+2n^*-n^* n_j-n_i}{2n^*-n^* n_j-n_i}>0$. Which completes the proof.
\end{proof}

To finish the proof, we need to show the negative part of Theorem \ref{thmDYH}.

\begin{proposition}\label{propUB}
 For any $p>1$, $R$ is not bounded from $L^{n^*,p}\to L^{n^*,p}$.   
\end{proposition}

\begin{proof}
By Proposition \ref{propHA} and Theorem \ref{thmLSI}, we know that
\begin{equation*}
    \|R_H(f)\|_{(n^*,p)} \lesssim \|f\|_{(n^*,p)}
\end{equation*}
Therefore, we only need to consider the low energy part, $R_L$. Thanks to \cite[Proposition 6.1]{HaSi}, after a series of simplifications, it suffices to show the operator:
\begin{equation*}
    f \mapsto a(z) \int_{E_i} b(z') f(z') d\mu_i(z')
\end{equation*}
where $a(z)$ is a function compactly supported on an end, say $E_j$, and not identically zero and 
\begin{equation*}
    b(z') = \int_0^1 (\Delta_i+k^2)^{-1}(z_i^0,z') \phi_i(z') dk \approx d(z_i^0,z')^{1-n_i}
\end{equation*}
does not bounded from $L^{n_i,p}$ to $L^{n^*,p}$ for $1\le i\le l$.

We prove it by giving a counterexample. Particularly, we consider function
\begin{equation*}
    f(z') = d(z_i^0,z')^{-1} \left(1+|\log [d(z_i^0,z')]|\right)^{-\beta} \quad 1/p<\beta \le 1
\end{equation*}
It is clear that
\begin{gather*}
    \int_{E_i} b(z') f(z') d\mu_i(z') \approx \int_{E_i} d(z_i^0,z')^{1-n_i} d(z_i^0,z')^{-1} \left(1+|\log [d(z_i^0,z')]|\right)^{-\beta} d\mu_i(z')\\
    \approx \int_1^\infty r^{-1} (1+|\log (r)|)^{-\beta} dr = \int_0^\infty (1+|x|)^{-\beta} dx = \infty
\end{gather*}
sicne $\beta \le 1$.

Hence, we only need to verify that $f\in L^{n_i,p}(E_i)$. Notice that $f$ is non-negative "radial" and decreasing. Set $f_0(r) = r^{-1}(1+|\log (r)|)^{-\beta}$ then we simply have $f(z') = f_0(d(z_i^0,z'))$. Moreover, since
\begin{equation*}
    f^*(t) = \inf \{s>0: d_f(s)\le t\}
\end{equation*}
and
\begin{gather*}
    d_f(s) = \mu_i \{z'\in E_i: |f(z')|>s\} = \mu_i \{z'\in E_i: f_0(d(z_i^0,z'))>s\}\\
    = \mu_i \{z'\in E_i: 1\le d(z_i^0,z')<f_0^{-1}(s)\} \lesssim [f_0^{-1}(s)]^{n_i}
\end{gather*}
Therefore
\begin{gather*}
    f^*(t) \lesssim \inf \{s>0: [f_0^{-1}(s)]^{n_i}\le t\} = \inf \{s>0: s\ge f_0(t^{1/{n_i}})\}\\
    = f_0(t^{1/{n_i}}) = t^{-1/{n_i}} (1+n_i^{-1} |\log (t)|)^{-\beta}
\end{gather*}
Consequently, we have
\begin{gather*}
    \|f\|_{(n_i,p)}^p = \int_0^\infty [t^{1/{n_i}}f^*(t)]^p \frac{dt}{t} \lesssim \int_0^\infty  (1+n_i^{-1} |\log (t)|)^{-\beta p} \frac{dt}{t}\\
    \approx \int_0^\infty (1+|x|)^{-p\beta} dx <\infty
\end{gather*}
since $\beta > 1/p$. As this argument applies to each end, the proof is complete.

\end{proof}

\section*{Acknowledgments} 
I would like to thank my supervisor Adam Sikora for introducing me to the research area discussed in the paper. I also want to thank Andrew Hassell for carefully reading the notes and giving precious suggestions.


\end{document}